\documentclass[11pt,a4paper]{article}
\usepackage{a4wide,amsfonts,amsmath,latexsym,amsthm,amssymb,euscript,eufrak,graphicx,units,mathrsfs,setspace,stmaryrd}
\usepackage[pdftex]{color}

\usepackage{float}
\newfloat{figure}{H}{lof}
\floatname{figure}{\figurename}

\usepackage[french,english]{babel}
\usepackage[T1]{fontenc}
\usepackage[utf8]{inputenc}
\usepackage{empheq}

\DeclareMathAlphabet{\eufrak}{U}{}{}{}  % Euler fraktur math
\SetMathAlphabet\eufrak{normal}{U}{euf}{m}{n}
\SetMathAlphabet\eufrak{bold}{U}{euf}{b}{n}

\newtheorem{prop}{Proposition}[section]

\newtheorem{theorem}[prop]{Theorem}
\newtheorem{lemma}[prop]{Lemma}

\newtheorem{assumption}[prop]{Assumption}

\theoremstyle{definition}

\newtheorem{remark}[prop]{Remark}

\newtheorem{definition}[prop]{Definition}

\newtheorem{notation}[prop]{Notation}
\newtheorem{convention}[prop]{Convention}

\numberwithin{equation}{section}

\def\E{\mathbb{E}}
\def\P{\mathbb{P}}
\def\real{\mathbb{R}}

\def\F{\mathcal{F}}
\def\1{\textbf{1}}
\def\ind#1{\textbf{1}_{\left\{#1\right\}}}

\def\X{\mathbb{X}}

\def\Cov{\textrm{Cov}}

\newcommand{\be}{\begin{equation}}
\newcommand{\ee}{\end{equation}}
\newcommand{\bde}{\begin{displaymath}}
\newcommand{\ede}{\end{displaymath}}
\newcommand{\beq}{\begin{eqnarray*}}
\newcommand{\eeq}{\end{eqnarray*}}
\newcommand{\beqa}{\begin{eqnarray}}
\newcommand{\eeqa}{\end{eqnarray}}
\newcommand{\bel }{\left\{\begin{array}{ll}}
\newcommand{\eel}{\cr \end{array} \right.}

\newcommand{\bex}{\begin{ex} \rm }
\newcommand{\eex}{\end{ex}}
{

\def\E{\mathbb E}
\def\F{{\cal F}}

\def\P{\mathbb P}

\def\I{\mathcal I}

\def\cal#1{\mathcal{#1}}

\author{Caroline Hillairet\footnote{ENSAE  Paris, CREST UMR 9194,
5  avenue Henry Le Chatelier
91120 Palaiseau, France.  Email: \texttt{caroline.hillairet@ensae.fr}} \and Anthony R\'eveillac\footnote{INSA de Toulouse, IMT UMR CNRS 5219, Universit\'e de Toulouse, 135 avenue de Rangueil 31077 Toulouse Cedex 4 France. \; Email: \texttt{anthony.reveillac@insa-toulouse.fr}} }

\title{Explicit correlations for the Hawkes processes\footnote{This research is supported by a grant of the French National Research Agency (ANR), Investissements d'Avenir (LabEx Ecodec/ANR-11-LABX-0047) and the Joint Research Initiative "Cyber Risk : actuarial modeling" with the partnership of AXA Research Fund.}}

\begin{document}

\maketitle

\allowdisplaybreaks

\begin{abstract}
\noindent
In this paper we fill a gap in the literature by providing exact and explicit expressions for the correlation of general Hawkes processes together with its intensity process.  Our methodology relies on the Poisson imbedding representation and on 
recent findings on Malliavin calculus and pseudo-chaotic representation for counting processes.
\end{abstract}


\section{Introduction}
\label{section:intro}

Hawkes processes have been introduced in \cite{Hawkes} to describe seismological phenomena with clustering features like earthquakes. 
Indeed, these point processes have the peculiarity to model excitation effects: past jumps impact the point process' intensity  through an excitation kernel   and thus trigger (or  inhibit if the kernel is negative) future jumps.  This implies correlations between successive jump events.
Since  their first historical  application in seismology, Hawkes processes have been widely used in many different fields, among which finance and insurance, neurosciences,   or social network modeling. 
Hawkes process are used  for instance in neurosciences to model the interactions between the neurons within their sequences of spikes   (see e.g. \cite{Evahawkes}),   to model retweet cascades   in social media (see e.g. \cite{rizoiu2017hawkes}), to model the arrival of defaults  in credit risk (see e.g.  \cite{Errais_et_al_2010}), or the arrival of sell/buy orders in  limit order book  for high-frequency finance (see e.g. \cite{bacry2015hawkes}); they can also be relevant to model the frequency component in insurance loss portfolios  (see e.g. \cite{Hillairet_Reveillac_Rosenbaum} for cyber-risk).

\noindent Mathematically, given a parameter $\mu>0$ and a mapping (often called the excitation kernel) $\Phi : \real_+ \to \real_+$, a Hawkes process with parameter $(\mu,\Phi)$ denotes a counting process $(H_t)_{t \geq 0}$ whose stochastic intensity $(\lambda_t)_{t\geq}$ satisfies a Volterra type integral equation : 
\begin{equation}
\label{eq:introlambda}
\lambda_t = \mu + \int_{(0,t)} \Phi(t-s) dH_s, \quad t\geq 0.
\end{equation}
Under the classical assumption $\|\Phi\|_1 < 1$ it has been proved in \cite{Hawkes,Hawkes_71b} that this Volterra equation is well-posed. More precisely, $H$ becomes a stationary process by replacing Relation (\ref{eq:introlambda}) by
$$ \lambda_t = \mu + \int_{(-\infty,t)} \Phi(t-s) dH_s, \quad t \geq 0.$$
This stationary feature is indeed interesting by itself but also allows one to characterize uniquely the Hawkes process within the class of stationary processes through its so-called \textit {first and second-order statistics} (see \cite{Hawkes,Bacry:2014aa}) that is the expression of $\E[\lambda_t]=\E[\lambda_0]$ and of $\Cov(dH_s,dH_t)$ which in the stationary case only depends on the difference $t-s$. Characterization here has to be understood by the fact that the kernel $\Phi$ is the unique solution to the Volterra integral equation describing the measure  $\Cov(dH_s,dH_t)$ (see \cite{Bacry:2014aa}  for a precise statement). However, even in this stationary case, this knowledge is not sufficient to derive general expressions for mixed correlations like $\E[\lambda_s H_t]$. In addition, for some applications,  considering the process starting at $-\infty$ is not always   relevant: this calls for studying Hawkes processes with intensity functions given by Relation \eqref{eq:introlambda}, for which stationarity is indeed lost. Another line of research to obtain quantitative information on the distribution of a Hawkes process consists in benefiting from the immigration-birth representation of a Hawkes process as obtained in \cite{hawkes1974cluster}. More specifically Laplace transform of marginals $H_t$ can be  proved to satisfy once again an integral Volterra equation. This allows one to derive moments of marginals $H_t$ and  to give an analytic expression for the probabilities $\P[H_t=k]$. This result has been extended to related processes such as  compound Hawkes processes like for example  in \cite{Gaoetal,Errais_et_al_2010,Euch_Rosenbaum}. These relations are valid in the non-stationary framework that is with an intensity process of the form (\ref{eq:introlambda}) but they do not provide  similar information on the intensity process $\lambda$
 and on mixed correlations for $(H_s,\lambda_t)$ with possible different marginal times $s$ and $t$. We finally mention that specific information on the law of $H$ and $\lambda$ like moments can be obtained in the particular cases of exponential and Erlang kernel (and of linear combinations of them). In a nutshell, these kernels write down as $\Phi(u):= \alpha u^n e^{-\beta u}$ for parameters  $\alpha,\beta,n$ to be chosen appropriately. The specific feature of these kernels lies in the fact that they share a Markovian structure for which  the so-called Dynkin formula can be used. In this line of research we mention \cite{Errais_et_al_2010,Cui_et_al_2020,Cui_et_al_2022,Privault_cumulants} to cite a few.\\\\
\noindent
In this paper, we fill this gap by providing in Theorem \ref{th:main} explicit expressions of quantities $\E[H_s H_t]$, $\E[\lambda_s H_t]$ and $\E[\lambda_s \lambda_t]$ for a general Hawkes process with general kernel $\Phi$. Our approach relies on recent findings on Malliavin calculus and \textit{pseudo}-chaotic representation for counting process obtained in \cite{Hillairet_Reveillac_Rosenbaum,Caroline_Anthony_chaotic}. As a by-product our methodology could apply to more general counting processes in the line of Theorem \ref{th:main2}. In particular we focus here on one-dimensional Hawkes processes and leave the extension to the multi-dimensional case for future research.\\\\ 
\noindent
We proceed as follows. Our main result Theorem \ref{th:main} is stated in Section \ref{section:main}. We present in Section \ref{section:preliminaries} the elements of Malliavin calculus for counting process that will be applied to the specific case of Hawkes processes in Section \ref{section:Hawkes}. Finally the proof of Theorem \ref{th:main} is given in Section \ref{section:proofmain}. 

\section{Main result}
\label{section:main}

Through this paper $\Phi$ denotes a map $\Phi:\real_+\to \real_+$.

\begin{assumption}
\label{assumption:Phi}
The mapping $\Phi : \real_+ \to \real_+$ belongs to $L^1(\real_+;dt)$ with 
$$\|\Phi\|_1:=\int_{\real_+} \Phi(t) dt < 1.$$
\end{assumption}

\begin{definition}[Hawkes process, \cite{Hawkes}]
\label{def:standardHawkes}
Let $(\Omega,\mathcal F,\P,\mathbb F:=(\mathcal F_t)_{t\geq 0})$ be a filtered probability space, $\mu>0$ and $\Phi:\real_+ \to \real_+$ satisfying Assumption \ref{assumption:Phi}. A Hawkes process $H:=(H_t)_{t\geq 0}$ with parameters $\mu$ and $\Phi$ is a counting process such that   
\begin{itemize}
\item[(i)] $H_0=0,\quad \P-a.s.$,
\item[(ii)] its ($\mathbb{F}$-predictable) intensity process is given by
\begin{equation}
\label{eq:lambda}
\lambda_t:=\mu + \int_{(0,t)} \Phi(t-s) dH_s, \quad t\geq 0,
\end{equation}
that is for any $0\leq s \leq t$ and $A \in \mathcal{F}_s$,
$$ \E\left[\textbf{1}_A (H_t-H_s) \right] = \E\left[\int_{(s,t]} \textbf{1}_A \lambda_r dr \right].$$
\end{itemize}
\end{definition}

\noindent By definition a Hawkes process exhibits a convolution structure related to Volterra integral equations as we will make precise in Section \ref{section:Volterra}. Similarly to Volterra ODEs, our expressions of correlations only involve the mapping $\Psi$ below which is the infinite sum of iterated convolutions of the excitation kernel $\Phi$.
 
\begin{prop}[See \textit{e.g.} \cite{Bacryetal2013}]
\label{prop:Phin}
Assume $\Phi$ enjoys Assumption \ref{assumption:Phi}. Let the sequence of iterated convolutions of $\Phi$ :
\begin{equation}
\label{eq:Phin}
\Phi_1:=\Phi, \quad \Phi_n(t):=\int_0^t \Phi(t-s) \Phi_{n-1}(s) ds, \quad t \in \real_+, \; n\in \mathbb{N}^*.
\end{equation}
 For every $n\geq 1$, $\|\Phi_n\|_1 = \|\Phi\|_1^n$ and the mapping 
\begin{equation}
\label{eq:Psi}
\Psi:=\sum_{n=1}^{+\infty} \Phi_n
\end{equation}
is well-defined as a limit in $L_1(\real_+;dt)$ and $\|\Psi\|_1 = \frac{\|\Phi\|_1}{1-\|\Phi\|_1}$.\\
\end{prop}

\newpage

\noindent
We now state our main result. 
\begin{theorem}
\label{th:main}
Let $(H_t)_{t\geq 0}$ be a Hawkes process (with intensity $\lambda$ satisfying (\ref{eq:lambda})) with parameters $\mu>0$ and $\Phi:\real_+ \to \real_+$ satisfying Assumption \ref{assumption:Phi}. \\

\noindent For any $t\geq 0$, it holds that 
\begin{equation}
\label{eq:expectation}
\left\lbrace
\begin{array}{l}
\E\left[H_t \right] = \mu \int_0^t\left(1+\int_0^u \Psi(r) dr\right) du \\ \\
\E\left[\lambda_t\right] = \mu \left(1+ \int_0^t \Psi(r) dr\right).
\end{array}
\right.
\end{equation}

\noindent For any $s,t$ with $0\leq s \leq t$, 
\begin{itemize}
\item[(i)] The covariance of the Hawkes process $H$ is given by
\begin{eqnarray}
\label{eq:covH}
&&\hspace*{-1cm}\Cov(H_s,H_t) =\E\left[H_s H_t \right] -\E\left[H_s\right] \E\left[H_t\right] \nonumber\\
&=&\mu \int_0^s \left(1+\int_0^{v} \Psi(w) dw\right) \left(1+\int_{v}^s \Psi(y-v) dy\right) \left(1+\int_v^{t} \Psi(y-v) dy\right) dv.\nonumber\\
\end{eqnarray}
\item[(ii)] The covariance  of the Hawkes' intensity $\lambda$ is given by
\begin{eqnarray}
\label{eq:covlambdalambda}
 \Cov(\lambda_s,\lambda_t) &=&\E\left[\lambda_s \lambda_t\right] - \E\left[\lambda_s\right] \E\left[\lambda_t\right]\nonumber\\ 
&=&\int_0^s \Psi(s-v) \Psi(t-v) \left(1+\int_0^{v} \Psi(v-w) dw\right) dv.\nonumber\\
\end{eqnarray}
\item[(iii)] The mixed correlation between the Hawkes process and its intensity is given by
\begin{equation}\label{eq:covlambdaH1}
\left\lbrace
\begin{array}{ll}
\Cov(\lambda_s,H_t) &=\E[\lambda_s H_t] - \E[\lambda_s] \E[H_t] \\ 
&=\int_0^s \Psi(s-v) \left(1+\int_0^{v} \Psi(w) dw\right) \left(1+ \int_v^{t} \Psi(y-v) dy\right) dv, \\  \\
\Cov(H_s,\lambda_t) &=\E[H_s \lambda_t] - \E[H_s] \E[\lambda_t]\\ 
&=\int_0^s \Psi(t-v) \left(1+\int_0^{v} \Psi(w) dw\right) \left(1+\int_{v}^s \Psi(y-v) dy\right) dv.
\end{array}
\right.
\end{equation}
\end{itemize}
\end{theorem}

\noindent The proof of Theorem \ref{th:main} is presented in Section \ref{subsection:proofof main} and follows from Theorem \ref{th:main2} in Section \ref{sssection:proofmain2}.

\begin{remark}
In particular we recover the expression of $\E[(H_t)^2]$ from \cite{Gaoetal} as 
$$ \E[(H_t)^2] = \left(\mu \int_0^t \Psi_1(s) ds\right)^2 + \mu \int_0^t \Psi_2(u) du,$$
where using notations of \cite{Gaoetal}, $\mu \int_0^t \Psi_1(s) ds = \E[H_t]$ and 
$$\mu \int_0^t \Psi_2(u) du =\mu \int_0^t \left(1+\int_0^v \Psi(w) dw\right) \left(1+\int_v^t \Psi(y-v) dy\right)^2 dv.$$
\end{remark}
 
\section{Elements of Malliavin calculus on the Poisson space}
\label{section:preliminaries}

We set $\mathbb{N}^*:=\mathbb{N} \setminus \{0\}$ the set of positive integers. We make use of the convention :

\begin{convention}
\label{convention:sums}
For $a, b \in \mathbb Z$ with $a > b$, and for any map $\rho : \mathbb Z \to \real$,
$$ \prod_{i=a}^b \rho(i) :=1; \quad \sum_{i=a}^b \rho(i) :=0.$$
\end{convention}
\noindent We set 
\begin{equation}
\label{eq:X}
\mathbb X:= \real_+\times \real_+ = \{x=(t,\theta), \; t \in \real_+, \; x\in \real_+\};
\end{equation}
Throughout this paper the notation $(t,\theta)$ will  refer to the first and second coordinate of an element in $\mathbb X$.

\begin{notation}
\label{notation:ordered}
Let $k\in \mathbb N^*$ and $(x_1,\ldots,x_k)=((t_1,\theta_1),\ldots,(t_k,\theta_k))$ in $\mathbb X^k$. We set $(x_{(1)},\ldots,x_{(k)})$ the ordered in the $t$-component of $(x_1,\ldots,x_k)$ with \;
$0 \leq t_{(1)} \leq \cdots \leq t_{(k)} ,$\;
and write $x_{(i)}:=(t_{(i)},\theta_{(i)})$. 
\end{notation}
\noindent We simply write $dx:=dt \, d\theta$ for the Lebesgue measure on $\mathbb X$. We also set $\mathcal B(\mathbb X)$ the set of Borelian of $\mathbb X$.\\\\
\noindent
Our approach lies on the so-called Poisson imbedding representation allowing one to represent a counting process with respect to a baseline random Poisson measure on $\mathbb X$. Most of the elements presented in this section are taken from \cite{Privault_2009,Last2016}.\\
\noindent 
We define $\Omega$ the space of configurations on $\mathbb X$ as 
$$ \hspace{-3.5em}\Omega:=\left\{\omega=\sum_{i=1}^{n} \delta_{x_i}, \; x_i:=(t_{i},\theta_i) \in \mathbb X,\; i=1,\ldots,n,\; 0=t_0 < t_1 < \cdots < t_n, \; \theta_i \in \mathbb{R}_+, \; n\in \mathbb{N}\cup\{+\infty\} \right\}.$$
Each path of a counting process is represented as an element $\omega$ in $\Omega$ which is a $\mathbb N$-valued $\sigma$-finite measure on $\mathbb X =\mathbb{R}_+^2$. Let $\mathcal F$ be the $\sigma$-field associated to the vague topology on $\Omega$. Let $\P$ the Poisson measure on $ \Omega$ under which the canonical process $N$ on $\Omega$ is a Poisson process with intensity one that is : 
$$ (N(\omega))([0,t]\times[0,b])(\omega):=\omega([0,t]\times[0,b]), \quad t \geq 0, \; b \in \mathbb{R}_+,$$
is an homogeneous Poisson process with intensity one ($N([0,t]\times[0,b])$ is a Poisson random variable with intensity $ b t$ for any $(t,b) \in \mathbb X$). We set $\mathbb F^N:=(\F_t^N)_{t\geq 0}$ the natural history of $N$, that is $\mathcal{F}_t^N:=\sigma\{N( \mathcal T  \times B), \; \mathcal T \subset \mathcal{B}([0,t]), \; B \in \mathcal{B}(\real_+)\}$. The expectation with respect to $\P$ is denoted by $\E[\cdot]$. We also set $\mathcal{F}_\infty^N:=\lim_{t \to +\infty} \mathcal{F}_t^N$.\\\\
\noindent
In order to introduce our add-points operators and the pathwise derivative we introduce some elements of stochastic analysis on the Poisson space. We set :
$$ L^0(\Omega):=\left\{ F:\Omega \to \real, \; \mathcal{F}_\infty^N-\textrm{ measurable}\right\},$$
$$ L^2(\Omega):=\left\{ F \in L^0(\Omega), \; \E[|F|^2] <+\infty\right\}.$$
Let for $j\in \mathbb{N}^*$
\begin{equation}
\label{definition:L2j}
L^2(\mathbb X^j) := \left\{f:\mathbb{X}^j \to \real, \; \int_{\mathbb{X}^j} |f(x_1,\cdots,x_j)|^2 dx_1 \cdots dx_j <+\infty\right\}.
\end{equation}

\begin{definition}[Symmetrization]
\label{defi:symm}
Let $j\in \mathbb N^*$.
\begin{itemize}
\item
For $f$ in $L^2(\mathbb X^j)$, we define $\tilde f$ the symmetrization of $f$ that is the map $\tilde f :\mathbb{X}^j \to \real$ defined as
\begin{equation}
\label{eq:symmetrization}
\tilde f(x_1,\cdots,x_j) := \frac{1}{j!} \sum_{\sigma \in \mathcal S_j} f(x_{\sigma(1)},\cdots,x_{\sigma(j)}),
\end{equation}
where $\mathcal S_j$ denotes the set of all bijections from $\{1,\cdots,j\}$ to itself.
\item A mapping $f$ in $L^2(\mathbb X^j)$ is  said symmetric if $f = \tilde f$ and we set
\begin{equation}
\label{definition:symm2}
L^2_s(\mathbb X^j) := \left\{f \in L^2(\mathbb X^j) \textrm{ and } f \textrm{ is symmetric} \right\}
\end{equation}
the set of symmetric square integrable functions $f$ on $\mathbb{X}^j$.
\end{itemize}
\end{definition}

\noindent
The main ingredient in this paper is  the add-points operators on the Poisson space $\Omega$. 
\begin{definition}$[$Add-points operators$]$\label{definitin:shifts}
\begin{itemize}
\item[(i)]
For $k$ in $\mathbb N^*$, and any subset of $\mathbb X$ of cardinal $k$ denoted $\{x_i, \; i\in \{1,\ldots,k\}\} \subset \mathbb X$, we set the measurable mapping :
\begin{eqnarray*}
\varepsilon_{(x_1,\ldots,x_k)}^{+,k} : \Omega & \longrightarrow & \Omega \\
     \omega & \longmapsto   & \omega + \sum_{i=1}^k \delta_{x_i};
\end{eqnarray*}
with the convention that given a representation of $\omega$ as $\omega=\sum_{i=1}^{n} \delta_{y_i}$ (for some $n\in \mathbb N^*$, $y_i \in \mathbb X$), $\omega + \sum_{i=1}^k \delta_{x_i}$ is understood as follows\footnote{Note that given fixed atoms $(x_1,\ldots,x_n)$, as $\P$ is the Poisson measure on $\Omega$, with $\P$-probability one, marks $x_i$ do not belong to the representation of $\omega$.} :
\begin{equation}
\label{eq:addjumpsum}
\omega + \sum_{i=1}^k \delta_{x_i} :=  \sum_{i=1}^{n} \delta_{y_i} + \sum_{i=1}^k \delta_{x_i} \ind{x_i \neq y_i}.
\end{equation}
\item[(ii)] When $k=1$ we simply write $\varepsilon_{x_1}^{+}:=\varepsilon_{x_1}^{+,1}$.
\end{itemize}
\end{definition}

In this paper we will also make use of a purely deterministic pathwise operator.  

\begin{definition}
\label{definition:patwisederivative}
Let $n\in \mathbb N^*$, and $(x_1,\cdots,x_n) \in \X^n$ with $t_1 < \cdots < t_n$. We set for $F\in L^1(\Omega)$,
$$ \mathcal D_{(x_1,\cdots,x_n)}^n F := \sum_{J\subset \{x_1,\cdots,x_n\}} (-1)^{n-|J|} F(\varpi_{J}),$$
where if $J=\{y_1,\ldots,y_k\}$, $\varpi_{\{y_1,\ldots,y_k\}} := \sum_{i=1}^k \delta_{y_i} \in \Omega$.
\end{definition}
\noindent In particular, even though $F$ is a random variable, $\mathcal D_{(x_1,\cdots,x_n)}^n F$ is a real number as each term $F(\varpi_{J})$ is the evaluation of $F$ at the outcome $\varpi_{J}$.\\\\
\noindent
The decompositions we are going to deal with take the form of iterated stochastic integrals whose definition is made precise in this section.

\begin{notation}
\label{notation:simplex}
For $j\in \mathbb N^*$,  we define the two following sets
\begin{eqnarray*}
\Delta_j&:=&\left\{(x_1,\ldots,x_j) \in \mathbb X^j, \; x_i \neq x_k, \; \forall i\neq k \in \{1,\cdots,j\}\right\},\\
\Delta_{(j)}&:=&\left\{(x_1,\ldots,x_j)=((t_1,\theta_1),\ldots,(t_j,\theta_j)) \in \mathbb X^j, \; t_1<\cdots <t_i<t_{i+1}<\cdots <t_j\right\}.
\end{eqnarray*}
\end{notation}
\newpage
\begin{definition}
\label{definition:interatedPoisson}
Let $j \in \mathbb N^*$.
\begin{itemize}
\item
For $f_j$ an element of $L^2(\mathbb X^j)$ (not necessarily symmetric) we set $\I_j(f_j)$ the $j$th iterated integral of $f_j$ against the Poisson measure defined as : 
$$ \I_j(f_j) :=  \int_{\Delta_j} f_j(x_1,\ldots,x_{j}) N(dx_{1}) \cdots N(dx_j)$$
where each of the integrals above is well-defined pathwise for $\P$-a.e. $\omega \in \Omega$ and where we recall the notation $x_i=(t_i,\theta_i)$ and $dx_i=dt_i \, d\theta_i$.
\item For $f_j$ in $L_s^2(\mathbb X^j)$ (that is a symmetric function according to Definition \ref{defi:symm}), the $j$th iterated integral above can be written as
\begin{align}
\label{eq:Incaln}
& \hspace{-2em} \I_j(f_j) \nonumber \\
&\hspace{-2em}= \int_{\Delta_j}  f_j(x_1,\ldots,x_{j}) N(dx_{1}) N(dx_j) \nonumber \\
&\hspace{-2em}=j! \int_{\mathbb X} \int_{[0,t_{j-1})\times \real_+} \cdots \int_{[0,t_{2})\times \real_+} f_j(x_1,\ldots,x_{j}) N(dx_{1}) \cdots N(dx_j) \nonumber \\
&\hspace{-2em}=j! \int_{\mathbb X} \int_{[0,t_{j-1})\times \real_+} \cdots \int_{[0,t_{2})\times \real_+} f_j((t_1,\theta_1),\ldots,(t_j,\theta_j)) N(d t_1,d\theta_1) \cdots N(d t_j,d\theta_j).
\end{align} 
\item By definition of the symmetrization (see once again Definition \ref{defi:symm}), for any $f_j$ element of $L^2(\mathbb X^j)$ (not necessarily symmetric), 
$$ \I_j(f_j) = \I_j(\tilde f_j). $$
\end{itemize}
\end{definition}

\begin{remark}
Note that our definition coincides with the notion of factorial measures as presented in \cite{Last2016}.
\end{remark}

\noindent We recall the pseudo-chaotic expansion as introduced in \cite{Caroline_Anthony_chaotic}. 

\begin{theorem}[Pseudo-chaotic expansion]
\label{th:chaoticandpseudochaotic}
Let $F$ in $L^2(\Omega)$. 
$F$ is said to admit a pseudo-chaotic expansion if there exists a sequence $(c_j^F)_{j\geq 1}$ with $c^F_j \in L^2_s(\mathbb X^j)$ (see Notation (\ref{definition:symm2})) such that : 
$$ F = \sum_{j=1}^{+\infty} \frac{1}{j!} \I_j(c_j^F).$$
According to \cite[Theorems 3.13 and 3.15]{Caroline_Anthony_chaotic} if such decomposition exists it is unique.
\end{theorem}
\noindent
We finally recall the following lemma which is a simple consequence of Mecke's formula with our notations (we refer to \textit{e.g.} Relation (11) in \cite{Last2016}). 

\begin{prop}
\label{prop:Mecke}
Let $j\in \mathbb N^*$, $c_j \in L^2_s(\mathbb X^j)$ and $F \in L^2(\Omega)$. Then
\begin{equation}
\label{eq:Mecke}
\E\left[F \; \I_j(c_j) \right] = \int_{\X^j} \E\left[F \circ \varepsilon_{(x_1,\ldots,x_j)}^{+,j} \right] c_j(x_1,\ldots,x_j) dx_1\cdots dx_j.
\end{equation}
In particular taking $F = 1$ we have that 
\begin{equation}
\label{eq:expectIj}
\E\left[\I_j(c_j) \right] = \int_{\X^j} c_j(x_1,\ldots,x_j) dx_1\cdots dx_j.
\end{equation}
\end{prop}

\section{Pseudo-chaotic expansion for the Hawkes process}
\label{section:Hawkes}

Our approach relies on a specific representation of the Hawkes process with respect to the enlarged Poisson noise $N$, see \textit{e.g.} \cite[Corollary 2.7]{Hillairet_Reveillac_Rosenbaum} known under the name of Poisson imbedding \cite{Bremaud_Massoulie}.

\begin{prop}
\label{prop:systemHawkesdeterministe}
Let $\Phi$ as in Assumption \ref{assumption:Phi} and $\mu>0$. The SDE below admits a unique solution $(\lambda_t)_{t\geq 0}$ :
\begin{equation}
\label{eq:EDOlambda}
\lambda_ t = \mu + \int_{(0,t)\times \real_+} \Phi(t-s) \ind{\theta\leq \lambda_s} N(ds,d\theta), \quad t \geq 0;
\end{equation}
and a Hawkes process  $H$  with intensity $\lambda$ (characterized by the parameters ($\mu, \Phi$))  can be represented as 
\begin{equation}
\label{eq:EDOH} 
H_t =\int_{(0,t]\times \real_+} \ind{\theta\leq \lambda_s} N(ds,d\theta), \quad t \geq 0.
\end{equation}
\end{prop}

\noindent
In order to perform our computations for both $H$ and $\lambda$ we collect them in a unique notation.

\begin{notation}
\label{notationX}
Given $\zeta \equiv \Phi$ or $\zeta \equiv 1$  we set : 
$$ X^{\zeta}_t := \int_{(0,t)\times \real_+} \zeta(t-s) \ind{\theta\leq \lambda_s} N(ds,d\theta), \quad t \geq 0,$$
$$
\mbox{ so that  } \quad X_t^\zeta = \left\lbrace \begin{array}{l} \lambda_t - \mu, \quad \textrm{ if } \zeta(u)=\Phi(u) \\ \\ H_t, \quad \textrm{ if } \zeta(u)=1. \\\end{array}\right.
$$
\end{notation}

\subsection*{Elements on Volterra integral equations}
\label{section:Volterra}

For $f,g$ in $L^1(\real_+;dt)$ we define the convolution of $f$ and $g$ by 
$$(f\ast g)(t):=\int_0^t f(t-u) g(u) du, \quad t \geq 0.$$
This allows one to solve a linear Volterra integral equation as follows.

\begin{lemma}[See Lemma 5 in \cite{Bacryetal2013}]
\label{lemma:Bacryetal}
For $g$ locally bounded, the unique solution $f_g$ to the equation: 
$$ f_g(t) = g(t) + \int_0^t \Phi(t-s) f_g(s) ds, \quad t\geq 0, $$
is given by 
$$ f_g(t) = g(t) + \int_0^t \Psi(t-s) g(s) ds, \quad t \geq0. $$ 
\end{lemma}

\noindent
We also recall \cite[Lemma 5.4]{Caroline_Anthony_chaotic}.

\begin{lemma}
\label{lemma:magic}
Let $f$ in $L_1(\real_+;dt)$. For any $n\in \mathbb{N}$ with $n \geq 3$, and for any $0\leq s \leq t$,
\begin{equation}
\label{eq:magic}
\int_s^t \int_s^u \Phi_{n-1}(u-r) f(r) dr du = \int_s^{t} \int_s^{v_n} \int_s^{v_{n-1}} \cdots \int_s^{v_2} \prod_{i=2}^{n} \Phi(v_{i}-v_{i-1}) f({v_1}) dv_1 \cdots dv_{n}.
\end{equation} 
In particular taking $f=1$, 
$$ \int_s^u \Phi_{n-1}(u-r) dr = \int_s^{u} \int_s^{v_{n-1}} \cdots \int_s^{v_2} \Phi(u-v_{n-1}) \prod_{i=2}^{n-1} \Phi(v_{i}-v_{i-1}) dv_1 \cdots dv_{n-1}, \quad \textrm{ for a.a. } u\in \real_+. $$
\end{lemma}

\begin{proof}
As mentioned, Relation (\ref{eq:magic}) is given in \cite[Lemma 5.4]{Caroline_Anthony_chaotic}. Set $M$ and $P$ the Borelian measures on $\real_+$,
$$M([0,t]):=\int_s^t \int_s^u \Phi_{n-1}(u-r) dr du;$$
$$P([0,t]):=\int_s^{t} \int_s^{v_n} \int_s^{v_{n-1}} \cdots \int_s^{v_2} \prod_{i=2}^{n} \Phi(v_{i}-v_{i-1}) f({v_1}) dv_1 \cdots dv_{n}.$$
Relation (\ref{eq:magic}) entails that $M=P$ leading to the equality in $L^1(\real_+)$ of their densities with respect to  the Lebesgue measure.  
\end{proof}

\subsection*{Pseudo-chaotic expansion of the Hawkes process}

We follow \cite{Caroline_Anthony_chaotic} to obtain the so-called pseudo-chaotic expansion for the Hawkes process. It relies on the iterated integrals $\I_n$ and on the pathwise derivative operators $\mathcal D^n$ respectively introduced in Definitions \ref{definition:interatedPoisson} and \ref{definition:patwisederivative}. 

\begin{prop}
\label{prop:pseudochageneral}
Let $\zeta \equiv \Phi$ or $\zeta \equiv 1$ and recall Notation \ref{notationX}. Let $t\geq 0$. Then $X^\zeta_t$ admits the pseudo-chaotic expansion below  
$$ X^\zeta_t = \sum_{n=1}^{+\infty} \frac{1}{n!} \I_n(c_n^{\zeta,t}),$$
with for all $ (x_1,\cdots,x_n) \in ([0,t]\times \real_+)^n$
$$ c_n^{\zeta,t}(x_1,\cdots,x_n)=\mathcal D_{(x_{(1)},\cdots,x_{(n)})}^n X_t^\zeta =\zeta(t-t_n) \mathcal D_{(x_{(1)},\cdots,x_{(n-1)})}^{n-1} \ind{\theta_{(n)} \leq \lambda_{(t_n)}}.$$  
\end{prop}

\begin{proof}
This follows from \cite{Caroline_Anthony_chaotic} which gives the pseudo-chaotic expansion of any random linear functional of $N$ restricted to a bounded domain (say $[0,T]\times[0,M]$, $T, M>0$) of $\real^2$; with a focus on random variables of the form $F=H_t$ where $H$ is a counting process with bounded intensity (we refer the reader to \cite{Caroline_Anthony_chaotic} for a complete exposition). Even though the intensity of a Hawkes process is unbounded, it is proved in \cite{Caroline_Anthony_chaotic} that marginals of Hawkes processes admit a pseudo-chaotic expansion. Mimicking this proof we get that 
$c_n^{\zeta,t}= \mathcal D^n X_t^\zeta$.  \\\\
\noindent
Let $(x_1,\cdots,x_n) \in ([0,t]\times \real_+)^n$ with $t_1 < \cdots <t_n$. Let $J\subset \{1,\cdots,n\}$. We have
\begin{eqnarray*}
X_t^\zeta(\sum_{k\in J} \delta_{x_k}) 
&= &\left(\int_{(0,t]\times \real_+} \zeta(t-s) \ind{\theta \leq \lambda_s} N(ds,d\theta)\right)\left(\sum_{k\in J} \delta_{x_k}\right) \\
&= &\sum_{k\in J} \zeta(t-t_k) \ind{\theta_k \leq \lambda_{t_k}(\sum_{j\in J} \delta_{x_j})} \\
&= &\sum_{k\in J} \zeta(t-t_k) \ind{\theta_k \leq \lambda_{t_k}(\sum_{j\in J \cap \{1,\cdots,k-1\}} \delta_{x_j})},
\end{eqnarray*}
as $\lambda$ is a predictable process. Hence using the notation $\{1,\cdots, k-1\}:=\emptyset$ for $k=1$,
\begin{align*}
&\mathcal D_{(x_1,\ldots,x_n)}^n X_t^\zeta\\
&= \sum_{J \subset \{1,\cdots,n\}} (-1)^{n-|J|} \left(\int_{(0,t]\times \real_+} \zeta(t-s) \ind{\theta \leq \lambda_s} N(ds,d\theta)\right)\left(\sum_{k\in J} \delta_{x_k}\right) \\
&= \sum_{J \subset \{1,\cdots,n\}} (-1)^{n-|J|} \sum_{k\in J} \zeta(t-t_k) \ind{\theta_k \leq \lambda_{t_k}(\sum_{j\in J \cap \{1,\cdots,k-1\}} \delta_{x_j})} \\
&= \sum_{k=1}^{n-1} \sum_{J \subset \{1,\cdots,n\}; k\in J} (-1)^{n-|J|} \zeta(t-t_k) \ind{\theta_k \leq \lambda_{t_k}(\delta_{x_k}+\sum_{j\in J \cap \{1,\cdots,k-1\}} \delta_{x_j})} \\
&+ \sum_{J \subset \{1,\cdots,n\}; n\in J} (-1)^{n-|J|} \zeta(t-t_k) \ind{\theta_n \leq \lambda_{t_n}(\delta_{x_n}+\sum_{j\in J \cap \{1,\cdots,n-1\}} \delta_{x_j})} -\ind{J=\emptyset} \times 0\\
&= \sum_{k=1}^{n-1} \zeta(t-t_k) \sum_{J \subset \{1,\cdots,k-1,k+1,\cdots,n\}} (-1)^{n-1-|J|} \ind{\theta_k \leq \lambda_{t_k}(\delta_{x_k}+\sum_{i\in J \cap \{1,\cdots,k-1\}} \delta_{x_i})} \\
&+\zeta(t-t_n)  \sum_{J \subset \{1,\cdots,n-1\}} (-1)^{n-1-|J|} \ind{\theta_n \leq \lambda_{t_n}(\delta_{x_n}+\sum_{i\in J \cap \{1,\cdots,n-1\}} \delta_{x_i})}.
\end{align*}
On the one hand, for $k\in \{1,\cdots,n-1\}$,
\begin{align*}
& \sum_{J \subset \{1,\cdots,k-1,k+1,\cdots,n\}} (-1)^{n-1-|J|} \ind{\theta_k \leq \lambda_{t_k}(\delta_{x_k}+\sum_{i\in J \cap \{1,\cdots,k-1\}} \delta_{x_i})}\\
&=\sum_{U=\tilde U\cup \{k\}; \tilde U\subset \{1,\cdots,k-1\}} \; \sum_{J \subset \{1,\cdots,k-1,k+1,\cdots,n\}; \tilde U \subset J} (-1)^{n-1-|J|} \ind{\theta_k \leq \lambda_{t_k}(\sum_{i\in U} \delta_{x_i})} \\\
&= \sum_{U=\tilde U\cup \{k\}; \tilde U\subset \{1,\cdots,k-1\}} \ind{\theta_k \leq \lambda_{t_k}(\sum_{i\in U} \delta_{x_i})} (-1)^{n-1-|\tilde U|}  \sum_{J \subset \{1,\cdots,k-1,k+1,\cdots,n\}; \tilde U \subset J} (-1)^{|\tilde U|-|J|} \\
&=0,
\end{align*}
as $[k \leq n-1]$  implies that $\#\{J \subset \{1,\cdots,k-1,k+1,\cdots,n\}; \tilde U \subset J \} >1$ and thus by Newton's Binomial formula \vspace*{-0.5cm}$$\sum_{J \subset \{1,\cdots,k-1,k+1,\cdots,n\}; \tilde U \subset J} (-1)^{|\tilde U|-|J|}=0.$$
On the other hand 
$$\zeta(t-t_n)  \sum_{J \subset \{1,\cdots,n-1\}} (-1)^{n-1-|J|} \ind{\theta_n \leq \lambda_{t_n}(\sum_{i\in J \cap \{1,\cdots,n-1\}} \delta_{x_i})} 
= \zeta(t-t_n) \mathcal D_{(x_1,\cdots,x_{n-1})}^{n-1} \ind{\theta_n \leq \lambda_{t_n}};$$
which concludes the proof.
\end{proof}

\begin{prop}
\label{prop:integratec}
Let $\zeta \equiv \Phi$ or $\zeta \equiv 1$ and recall Notation \ref{notationX}. Let $t\geq 0$, $n\in \mathbb N^*$ and $(x_1,\cdots,x_n) \in \mathbb X^n$ with $0< t_1 < \ldots < t_n \leq t$. We have that :
$$ \int_{\real_+^n} \ldots \int c_n^{\zeta,t}(x_1,\cdots,x_n) d \theta_1 \ldots d\theta_n = \mu \zeta(t-t_n) \prod_{i=2}^{n} \Phi(t_i-t_{i-1}).$$
\end{prop}

\begin{proof}
Note first that as $\mathcal D^k \mu =0$, we proved in the proof of Proposition \ref{prop:pseudochageneral} that :
$$ \mathcal D_{(x_1,\ldots,x_n)}^n \lambda_s = \Phi(s-t_n) \mathcal D_{(x_1,\cdots,x_{n-1})}^{n-1} \ind{\theta_n \leq \lambda_{t_n}}; \quad \forall s \geq t_n.$$

We have that : 
\begin{eqnarray*}
\int_{\real_+} c_n^{\zeta,t}(x_1,\cdots,x_n) d\theta_n 
&= &\zeta(t-t_n) \int_{\real_+} \mathcal D^{n-1}_{x_1,\cdots,x_{n-1}} \ind{\theta_n \leq \lambda_n} d\theta_n \\
&=& \zeta(t-t_n) \int_{\real_+}  \sum_{J \subset \{1,\cdots,n-1\}} (-1)^{n-1-|J|} \ind{\theta_n \leq \lambda_{t_n}\left(\sum_{j\in J} \delta_{x_j}\right)} d\theta_n \\
&= &\zeta(t-t_n) \sum_{J \subset \{1,\cdots,n-1\}} (-1)^{n-1-|J|}  \int_{\real_+} \ind{\theta_n \leq \lambda_{t_n}\left(\sum_{j\in J} \delta_{x_j}\right)} d\theta_n \\
&= &\zeta(t-t_n) \sum_{J \subset \{1,\cdots,n-1\}} (-1)^{n-1-|J|}  \lambda_{t_n}\left(\sum_{j\in J} \delta_{x_j}\right) \\
&= &\zeta(t-t_n) \mathcal D^{n-1}_{x_1,\cdots,x_{n-1}} \lambda_{t_n}\\
&= &\begin{cases} \zeta(t-t_1) \mu, \quad \textrm{ if } n=1\\ \zeta(t-t_n) \Phi(t_n-t_{n-1}) D^{n-2}_{x_1,\cdots,x_{n-2}} \ind{\theta_{n-1} \leq \lambda_{t_{n-1}}}, \quad \textrm{ if } n\geq 2\\\end{cases}\\
&= &\begin{cases} \zeta(t-t_1) \mu, \quad \textrm{ if } n=1\\ \zeta(t-t_n) \Phi(t_n-t_{n-1}) c_{n-1}^{\zeta,t}(x_1,\cdots,x_{n-1}), \quad \textrm{ if } n\geq 2.\end{cases}\\
\end{eqnarray*}
The result follows by induction.
\end{proof}

\begin{remark}
By Relation (\ref{eq:expectIj}) we immediately get that 
$$ \E[X_t^\zeta] = \sum_{n\geq 1} \frac{1}{n!} \int_{\X^n} c_n^{\zeta,t}(x_1,\cdots,x_n) dx_1 \ldots x_n$$
which will allow us to recover the well-known expressions of $\E[H_t]$ and of $\E[\lambda_t]$ (see the proof of Theorem \ref{th:main2}).
\end{remark}

\section{Proof of Theorem \ref{th:main}}
\label{section:proofmain}

The proof relies on several results. 

\subsection{Preliminary results}\label{sssection:proofmain2}

The lemma below is a key observation on the support of the coefficients $c_n^{\zeta,t}$.

\begin{lemma}
\label{lemma:necessarycn}
Let $\zeta \equiv \Phi$ or $\zeta \equiv 1$ and recall Notation \ref{notationX}. Fix $t\geq 0$, let $n\in \mathbb N^*$, $(x_1,\ldots,x_n) \in \Delta_{(n)}\cap ((0,t]\times \real_+)^n$. It holds that : 
$$ c_n^{\zeta,t}(x_1,\ldots,x_n) = c_n^{\zeta,t}(x_1,\ldots,x_n)  \ind{\theta_1 \leq \mu} \prod_{i=2}^{n} \ind{\theta_i \leq \mu+\sum_{j=1}^{i-1} \Phi(t_i-t_j)}.$$
\end{lemma}

\begin{proof}
By definition
$$ c_n^{\zeta,t}(x_1,\ldots,x_n) = \zeta(t-t_n) \sum_{J\subset \{x_1,\cdots,x_{n-1}\}} (-1)^{n-1-|J|} \ind{\theta_n \leq \lambda_{t_n}(\sum_{j \in J} \delta_j)}.$$
Hence
$$ c_n^{\zeta,t}(x_1,\ldots,x_n) \ind{\theta_n > \mu+\sum_{j=1}^{n-1} \Phi(t_n-t_j)}=0. $$
Let $k \in \{1,\cdots,n-1\}$.
\begin{align*}
&c_n^{\zeta,t}(x_1,\ldots,x_n) \zeta(t-t_n)\ind{\theta_{k} > \mu+\sum_{j=1}^{k} \Phi(t_k-t_j)} \\
&= \zeta(t-t_n) \ind{\theta_{k} > \mu+\sum_{j=1}^{k} \Phi(t_k-t_j)} \sum_{J\subset \{x_1,\cdots,x_{n-1}\}} (-1)^{n-1-|J|} \ind{\theta_n \leq \lambda_{t_n}(\sum_{j\in J} \delta_j)}\\
&= \zeta(t-t_n) \ind{\theta_{k} > \mu+\sum_{j=1}^{k} \Phi(t_k-t_j)}\sum_{J\subset \{x_1,\cdots,x_{n-1}\}; x_{k} \in J} (-1)^{n-1-|J|} \ind{\theta_n \leq \lambda_{t_n}(\sum_{j \in J} \delta_j)}\\
&+ \zeta(t-t_n) \ind{\theta_{k} > \mu+\sum_{j=1}^{k} \Phi(t_k-t_j)} \sum_{J\subset \{x_1,\cdots,x_{n-1}\}; x_k \notin J} (-1)^{n-1-|J|} \ind{\theta_n \leq \lambda_{t_n}(\sum_{j \in J} \delta_j)}\\
&= \zeta(t-t_n) \ind{\theta_{k} > \mu+\sum_{j=1}^{k} \Phi(t_k-t_j)} \sum_{J\subset \{x_2,\cdots,x_{k-1},x_{k+1},\cdots,x_{n-1}\}} (-1)^{n-2-|J|} \ind{\theta_n \leq \lambda_{t_n}(\delta_{x_{k}}+\sum_{j \in J} \delta_j)}\\
&+ \zeta(t-t_n) \ind{\theta_{k} > \mu+\sum_{j=1}^{k} \Phi(t_k-t_j)} \sum_{J\subset \{x_2,\cdots,x_{n-1}\}} (-1)^{n-1-|J|} \ind{\theta_n \leq \lambda_{t_n}(\sum_{j \in J} \delta_j)}\\
&=0,
\end{align*}
as $ \lambda_{t_n}(\delta_{x_k} + \sum_{j \in J} \delta_j) \ind{\theta_{k} > \mu+\sum_{j=1}^{k} \Phi(t_k-t_j)} = \lambda_{t_n}(\sum_{j \in J} \delta_j) \ind{\theta_{k} > \mu+\sum_{j=1}^{k} \Phi(t_k-t_j)}$.
\end{proof}
\noindent Based on the previous observation, the expectation of the perturbed intensity only depends on the time-components of the marks $x_i$ provided the $\theta_i$ parameters  belong to the support described above. This constitutes a sort of decoupling of the $t_i$ and $\theta_i$ components.

\begin{prop}
\label{prop:expectlambdashifted}
Fix $t\geq 0$, $(x_1,\ldots,x_n) \in \Delta_{(n)}$ such that  \vspace*{-0.3cm}
$$ \ind{\theta_1 \leq \mu} \prod_{i=2}^{n} \ind{\theta_i \leq \mu+\sum_{j=1}^{i-1} \Phi(t_i-t_j)} = 1.$$
\vspace*{-0.3cm}
\begin{align*}
{\mbox Then } \quad \E\left[\lambda_t\circ \varepsilon_{(x_1,\ldots,x_{n})}^{+,n}\right] = \mu \left(1+\int_0^t \Psi(t-s) ds\right) +\sum_{j=1}^n \Psi(t-t_j) \textbf{1}_{[t_j,+\infty)}(t), \quad t \geq 0.
\end{align*}
\end{prop}

\begin{proof}
Let $t\geq 0$ and $(x_1,\ldots,x_n) \in \Delta_{(n)}$ such that $\ind{\theta_1 \leq \mu} \prod_{i=2}^{n} \ind{\theta_i \leq \mu+\sum_{j=1}^{i-1} \Phi(t_i-t_j)} = 1$. We have that
\begin{align*}
\E\left[\lambda_t\circ \varepsilon_{(x_1,\ldots,x_{n})}^{+,n}\right] 
= &\left[\mu+\E\left[\int_{0}^{t\wedge t_{1}} \Phi(t-s) \lambda_s\circ \varepsilon_{(x_1,\ldots,x_{n})}^{+,n} ds\right]\right]\nonumber\\
&+ \sum_{i=1}^{n-1} \ind{t\geq t_i} \ind{\theta_i \leq \lambda_{t_i}} \left[ \Phi(t-t_i) + \E\left[\int_{t_i}^{t\wedge t_{i+1}} \Phi(t-s) \lambda_s\circ \varepsilon_{(x_1,\ldots,x_{n})}^{+,n} ds \right]\right]\nonumber \\
&+ \ind{t\geq t_{n-1}} \ind{\theta_i \leq \lambda_{t_{n-1}}} \left[\Phi(t-t_{n-1}) + \E\left[\int_{t_{n-1}}^{t}\Phi(t-s) \lambda_s\circ \varepsilon_{(x_1,\ldots,x_{n})}^{+,n} ds\right] \right].
\end{align*}
Fix $i$. By assumption $\prod_{j=1}^i \ind{\theta_j \leq \mu+\sum_{k=1}^{j-1} \Phi(t_k-t_j)}=1$, hence by definition of $\lambda_{t_i}$ we have that $\lambda_{t_i} \geq \mu+\sum_{k=1}^{j-1} \Phi(t_k-t_j)$ leading to $ \ind{\theta_i \leq \lambda_{t_i}} \prod_{j=1}^i \ind{\theta_j \leq \mu+\sum_{k=1}^{j-1} \Phi(t_k-t_j)} = 1.$ Coming back to our computation,

\begin{align*}
\E\left[\lambda_t\circ \varepsilon_{(x_1,\ldots,x_{n})}^{+,n}\right] 
=& \left[\mu+\int_{0}^{t\wedge t_{1}} \Phi(t-s) \E\left[\lambda_s\circ \varepsilon_{(x_1,\ldots,x_{n})}^{+,n} \right] ds\right]\nonumber\\
&+ \sum_{i=1}^{n-1} \ind{t\geq t_i} \left[ \Phi(t-t_i) + \int_{t_i}^{t\wedge t_{i+1}} \Phi(t-s) \E\left[\lambda_s\circ \varepsilon_{(x_1,\ldots,x_{n})}^{+,n} ds \right]\right]\nonumber \\
&+ \ind{t\geq t_{n}} \left[\Phi(t-t_{n-1})+ \int_{t_{n-1}}^{t}\Phi(t-s) \E\left[\lambda_s\circ \varepsilon_{(x_1,\ldots,x_{n})}^{+,n} ds\right] \right]\\
&= \mu^{t_1,\cdots,t_{n}}(t)+\int_{0}^{t} \Phi(t-s) \E\left[\lambda_s\circ \varepsilon_{(x_1,\ldots,x_{n})}^{+,n}\right] ds,
\end{align*}
where 
$$\mu^{t_1,\cdots,t_{n}}(t) := \mu + \sum_{i=1}^{n} \ind{t\geq t_i} \Phi(t-t_i).$$ 
We recognize the ODE in  Lemma \ref{lemma:Bacryetal} whose unique solution is,  for  $t \geq 0$
$$\mu^{t_1,\cdots,t_{n}}(t) + \int_0^t \Psi(t-s) \mu^{t_1,\cdots,t_{n}}(s)ds= \mu \left(1+\int_0^t \Psi(t-s) ds\right) +\sum_{j=1}^n \Psi(t-t_j) \ind{[t_j,+\infty)}(t) .
$$
\end{proof}
\noindent The previous formula propagates to the process  $X^\zeta$ as follows.

\begin{prop}
\label{prop:expectationXtshifted}
Let $\zeta \equiv \Phi$ or $\zeta \equiv 1$ and recall Notation \ref{notationX}. $$\hspace*{-3.5cm}  \mbox{ Let } (x_1,\ldots,x_n) \in \Delta_{(n)}  \mbox{ such that  }
 \ind{\theta_1 \leq \mu} \prod_{i=2}^{n} \ind{\theta_i \leq \mu+\sum_{j=1}^{i-1} \Phi(t_i-t_j)} = 1.$$
Let $t\geq 0$ such that $t_1<\cdots<t_n\leq t$.
It holds that 
$$ \E\left[X_t^{\zeta}\circ \varepsilon_{(x_1,\ldots,x_{n})}^{+,n}\right] =\int_0^{t} \zeta(t-u) \varphi^{t_1,\cdots,t_n}(u)du + \sum_{i=1}^{n} \zeta(t-t_i)$$
\begin{equation} \label{eq:mushifted}
\mbox{ with  } \quad 
\varphi^{t_1,\cdots,t_n}(u) := \mu \left(1+\int_0^u \Psi(u-v) dv\right) +\sum_{j=1}^n \Psi(u-t_j) \ind{[t_j,+\infty)}(u), \quad u \geq 0.
\end{equation}
\end{prop}

\begin{proof}
We have 
\begin{align*}
&X_t^{\zeta} \circ \varepsilon_{(x_1,\ldots,x_{n})}^{+,n} \\
&= X_{t\wedge (t_1-)}^{\zeta} + \ind{\theta_1 \leq \mu} \ind{\theta_1 \leq \lambda_{t_1}}  \ind{t\geq t_1} \left[\zeta(t-t_1) + \int_{(t_1,t]\times \real_+} \zeta(t-s) \ind{\theta\leq \lambda_s\circ \varepsilon_{(x_1,\ldots,x_{n})}^{+,n}} N(ds,d\theta)\right]\\
&= X_{t\wedge (t_1-)}^{\zeta} + \ind{t\geq t_1} \left[ \zeta(t-t_1) + \int_{(t_1,t\wedge t_2)\times \real_+} \zeta(t-s) \ind{\theta\leq \lambda_s\circ \varepsilon_{(x_1,\ldots,x_{n})}^{+,n}} N(ds,d\theta)\right] \\
&+ \ind{t\geq t_2} \ind{\theta_1 \leq \mu} \ind{\theta_2 \leq \mu + \Phi(t_2-t_1)} \ind{\theta_2 \leq \lambda_{t_2}\circ \varepsilon_{x_1}^{+}} \left[\zeta(t-t_2) +\int_{(t_2,t]\times \real_+} \zeta(t-s) \ind{\theta\leq \lambda_s\circ \varepsilon_{(x_1,\ldots,x_{n})}^{+,n}} N(ds,d\theta) \right]\\
&= X_{t\wedge (t_1-)}^\zeta 
+ \ind{t\geq t_1} \left[ \zeta(t-t_1) + \int_{(t_1,t\wedge t_2)\times \real_+} \zeta(t-s) \ind{\theta\leq \lambda_s\circ \varepsilon_{(x_1,\ldots,x_{n})}^{+,n}} N(ds,d\theta)\right] \\
&\hspace{-2em}+ \ind{t\geq t_2} \ind{\theta_1 \leq \mu} \ind{\theta_2 \leq \mu+ \Phi(t_2-t_1)} \ind{\theta_2 \leq \lambda_{t_2}\circ \varepsilon_{x_1}^{+}} \left[\zeta(t-t_2) +\int_{(t_2,t]\times \real_+} \zeta(t-s) \ind{\theta\leq \lambda_s\circ \varepsilon_{(x_1,\ldots,x_{n})}^{+,n}} N(ds,d\theta) \right],
\end{align*}
where $\ind{\theta_1 \leq \mu} \ind{\theta_2 \leq \mu+ \Phi(t_2-t_1)} \ind{\theta_2 \leq \lambda_{t_2}\circ \varepsilon_{x_1}^{+}} = 1$ following the same lines as in the previous proof.
Hence by induction we derive that 
\begin{align*}
X_t^{\zeta}\circ \varepsilon_{(x_1,\ldots,x_{n})}^{+,n} 
&= X_{t\wedge (t_1-)}^{\zeta} \\
&+ \sum_{i=1}^{n-1} \ind{t\geq t_i} \left[ \zeta(t-t_i) + \int_{(t_i,t\wedge t_{i+1})\times \real_+} \zeta(t-s) \ind{\theta\leq \lambda_s\circ \varepsilon_{(x_1,\ldots,x_{n})}^{+,n}} N(ds,d\theta)\right] \\
&+ \ind{t\geq t_{n}} \left[\zeta(t-t_n) + \int_{(t_{n},t]\times \real_+} \zeta(t-s) \ind{\theta\leq \lambda_s\circ \varepsilon_{(x_1,\ldots,x_{n})}^{+,n}} N(ds,d\theta) \right].
\end{align*}
Taking the expectation we get that : 
\begin{align*}
\E\left[X_t^{\zeta}\circ \varepsilon_{(x_1,\ldots,x_{n})}^{+,n}\right] 
&=\E\left[X_{t\wedge (t_1-)}^{\zeta}\right] \\
&+ \sum_{i=1}^{n-1} \ind{t\geq t_i} \left[ \zeta(t-t_i) + \int_{(t_i,t\wedge t_{i+1})} \zeta(t-s) \E\left[\lambda_s\circ \varepsilon_{(x_1,\ldots,x_{n})}^{+,n}\right] ds\right] \\
&+ \ind{t\geq t_{n}} \left[\zeta(t-t_n) + \int_{(t_{n},t]} \zeta(t-s) \E\left[\lambda_s\circ \varepsilon_{(x_1,\ldots,x_{n})}^{+,n}\right] ds\right] \\
&=\int_0^{t\wedge t_1} \zeta(t-s) \E\left[\lambda_s\circ \varepsilon_{(x_1,\ldots,x_{n})}^{+,n}\right] ds \\
&+ \sum_{i=1}^{n-1} \ind{t\geq t_i} \left[ \zeta(t-t_i) + \int_{(t_i,t\wedge t_{i+1})} \zeta(t-s) \E\left[\lambda_s\circ \varepsilon_{(x_1,\ldots,x_{n})}^{+,n}\right] ds\right] \\
&+ \ind{t\geq t_{n}} \left[\zeta(t-t_n) + \int_{(t_{n},t]} \zeta(t-s) \E\left[\lambda_s\circ \varepsilon_{(x_1,\ldots,x_{n})}^{+,n}\right] ds\right].
\end{align*}
Since we assume $t\geq t_n$,
$$\E\left[X_t^{\zeta}\circ \varepsilon_{(x_1,\ldots,x_{n})}^{+,n}\right] =\int_0^{t} \zeta(t-s) \E\left[\lambda_s\circ \varepsilon_{(x_1,\ldots,x_{n})}^{+,n}\right] ds + \sum_{i=1}^{n} \zeta(t-t_i).
$$
The conclusion follows by Proposition \ref{prop:expectlambdashifted}.
\end{proof}

\begin{theorem}
\label{th:main2}
Let $\zeta \equiv \Phi$ or $\zeta \equiv 1$ (recall Notation \ref{notationX}). 
\begin{itemize}
\item[(i)] For $s\geq 0$
$$
\E\left[X_s^\zeta\right] =\left\lbrace\begin{array}{l} \mu s + \mu \int_0^s\int_0^u \Psi(r) dr du, \quad \textrm{ if } \zeta \equiv 1,\\ \\ \mu \int_0^s \Psi(u) du, \quad \textrm{ if } \zeta \equiv \Phi. \end{array}\right..
$$
\item[(ii)] 
Let $s\leq t$.
We have 
\begin{align}
\label{eq:mainbetaphi}
\E\left[X_s^\Phi X_t^\zeta\right] 
&= \E\left[X_s^\Phi\right] \E\left[X_t^\zeta\right] 
+\mu \int_0^s \zeta(t-v) \Psi(s-v) \left(1+\int_0^{v} \Psi(w) dw\right) dv \nonumber \\
&+\mu \int_0^t \int_0^{s\wedge u} \zeta(t-u) \Psi(u-v) \Psi(s-v) \left(1+\int_0^{v} \Psi(v-w) dw\right) dv du.
\end{align}
\begin{align}
\label{eq:mainbeta1}
\E\left[X_s^1 X_t^\zeta\right] 
&= \E\left[X_s^1\right] \E\left[X_t^\zeta\right] 
+\mu \int_0^s \zeta(t-u) \left( 1+\int_{u}^s \Psi(y-u) dy\right) \left(1+\int_0^{u} \Psi(w) dw\right) du \nonumber \\
&+\mu \int_0^t \int_0^{s\wedge u} \zeta(t-u) \Psi(u-v) \left(1+\int_0^{v} \Psi(w) dw\right) \left(1+\int_{v}^s \Psi(y-v) dy\right) dv du.
\end{align}

\end{itemize}
\end{theorem}

\begin{proof}
Through this proof we adopt the notations that 
$$ \int_a^b \Phi_0(t) dt :=1, \quad \forall (a,b), \; 0 \leq a \leq b,$$
$$\mbox{  and for  } n\in \mathbb N^*, \quad 
\Delta^s_{(n)}:=\left\{(x_1,\cdots,x_n) \in ([0,s]\times \real_+)^n; \; t_1<\cdots<t_n \leq s\right\}.$$
We start with Part (i). By the pseudo-chaotic expansion of $X_s^\zeta$, we have that 
$$ X_s^\zeta = \sum_{n=1}^{+\infty} \frac{1}{n!} \I_n(c_n^{\zeta,s}).$$
Hence using Relation (\ref{eq:expectIj}) and Proposition \ref{prop:Mecke}
\begin{align*}
\E\left[X_s^\zeta \right] &=\sum_{n=1}^{+\infty} \frac{1}{n!} \E\left[\I_n(c_n^{\zeta,s})\right]\\
&=\sum_{n=1}^{+\infty} \int_{\Delta^s_{(n)}} c_n^{\zeta,s}(x_1,\cdots,x_n) dx_1\cdots dx_n\\
&=\mu \sum_{n=1}^{+\infty} \int_{0 \leq t_1<\cdots<t_n\leq s} \zeta(s-t_n) \prod_{i=2}^{n} \Phi(t_i-t_{i-1}) dt_1\cdots dt_n,
\end{align*}
where the last equality follows from Proposition \ref{prop:integratec}. Lemma \ref{lemma:magic} gives that 
\begin{align*}
\E\left[X_s^1 \right] 
&=\mu \sum_{n=1}^{+\infty} \int_{0 \leq t_1<\cdots<t_n\leq s} \prod_{i=2}^{n} \Phi(t_i-t_{i-1}) dt_1\cdots dt_n \\
&= \mu \sum_{n=1}^{+\infty} \int_0^s\int_0^u \Phi_{n-1}(u-r) dr du \\
&=\mu \int_0^s\left(1+\int_0^u \Psi(r) dr\right) du.
\end{align*}
The second part of Lemma \ref{lemma:magic} implies that 
\begin{align*}
\E\left[X_s^\Phi \right] 
&=\mu \sum_{n=1}^{+\infty} \int_{0 \leq t_1<\cdots<t_n\leq s} \Phi(s-t_n) \prod_{i=2}^{n} \Phi(t_i-t_{i-1}) dt_1\cdots dt_n \\
&= \mu \sum_{n=1}^{+\infty} \int_0^s \Phi_{n}(s-r) dr= \mu \int_0^s \Psi(r) dr.
\end{align*}
We turn to Part (ii). Let $\xi \equiv \Phi$ or $\xi \equiv 1$.  Using Mecke's formula (\ref{eq:Mecke}), Lemma \ref{lemma:necessarycn} and Proposition \ref{prop:expectationXtshifted}, we get
\begin{align*}
&\E\left[X_s^\xi X_t^\zeta\right] \\
&= \sum_{n=1}^{+\infty} \frac{1}{n!} \E\left[\I_n(c_n^{\xi,s}) X_t^\zeta\right] \\
&= \sum_{n=1}^{+\infty} \frac{1}{n!} \int_{([0,s]\times \real_+)^n} \E\left[c_n^{\xi,s}(x_1,\cdots,x_n) X_t^\zeta \circ \varepsilon_{(x_1,\ldots,x_{n})}^{+,n} \right] dx_1\cdots dx_n\\
&= \sum_{n=1}^{+\infty} \int_{\Delta^s_{(n)}} c_n^{\xi,s}(x_1,\cdots,x_n) \E\left[ X_t^\zeta \circ \varepsilon_{(x_1,\ldots,x_{n})}^{+,n} \right] dx_1\cdots dx_n\\
&= \sum_{n=1}^{+\infty} \int_{\Delta^s_{(n)}} c_n^{\xi,s}(x_1,\cdots,x_n) \left[\int_0^{t} \zeta(t-u) \varphi^{t_1,\cdots,t_n}(u)du + \sum_{i=1}^{n} \zeta(t-t_i)\right] dx_1\cdots dx_n
\end{align*}
where $\varphi^{t_1,\cdots,t_n}$ is given by (\ref{eq:mushifted}). Thus using Proposition \ref{prop:integratec}, we get 
\begin{align*}
&\E\left[X_s^\xi X_t^\zeta\right] \\
&= \sum_{n=1}^{+\infty} \int_0^s \int_0^{t_n} \cdots \int_{0}^{t_2} \int_{\real_+^n} c_n^{\xi,s}(x_1,\cdots,x_n) d\theta_1 \cdots d\theta_n  \left[\int_0^{t} \zeta(t-u) \varphi^{t_1,\cdots,t_n}(u)du + \sum_{i=1}^{n} \zeta(t-t_i)\right] dt_1\cdots dt_n \\
&=  \mu \sum_{n=1}^{+\infty} \int_0^s \int_0^{t_n} \cdots \int_{0}^{t_2} \xi(s-t_n) \prod_{i=2}^{n} \Phi(t_i-t_{i-1}) \left[\int_0^{t} \zeta(t-u) \varphi^{t_1,\cdots,t_n}(u)du + \sum_{i=1}^{n} \zeta(t-t_i)\right] dt_1\cdots dt_n.
\end{align*}
Using the definition of $\varphi^{t_1,\cdots,t_n}$ given by (\ref{eq:mushifted}) that is 
$$ \varphi^{t_1,\cdots,t_n}(u) = \mu \left(1+\int_0^u \Psi(u-v) dv\right) +\sum_{j=1}^n \Psi(u-t_j) \ind{[t_j,+\infty)}(u); u\geq 0,$$
the previous expression can be written  as :
\begin{align*}
&\E\left[X_s^\xi X_t^\zeta\right] \\
&=  \mu \sum_{n=1}^{+\infty} \int_0^s \int_0^{t_n} \cdots \int_{0}^{t_2} \xi(s-t_n) \prod_{i=2}^{n} \Phi(t_i-t_{i-1}) \left[\int_0^{t} \zeta(t-u) \varphi^{t_1,\cdots,t_n}(u)du + \sum_{i=1}^{n} \zeta(t-t_i)\right] dt_1\cdots dt_n \\
&=  \mu^2 \sum_{n=1}^{+\infty} \int_0^s \int_0^{t_n} \cdots \int_{0}^{t_2} \xi(s-t_n) \prod_{i=2}^{n} \Phi(t_i-t_{i-1}) \left[\int_0^{t} \zeta(t-u) \left(1+\int_0^u \Psi(u-v) dv\right) du \right] dt_1\cdots dt_n \\
&+  \mu \sum_{n=1}^{+\infty} \int_0^s \int_0^{t_n} \cdots \int_{0}^{t_2} \xi(s-t_n) \prod_{i=2}^{n} \Phi(t_i-t_{i-1}) \\
&\times \left[\int_0^{t} \zeta(t-u) \sum_{j=1}^n \Psi(u-t_j) \ind{[t_j,+\infty)}(u) du + \sum_{i=1}^{n} \zeta(t-t_i)\right] dt_1\cdots dt_n \\
&= \mu^2 \left[\int_0^{t} \zeta(t-u) \left(1+\int_0^u \Psi(u-v) dv\right) du \right] \sum_{n=1}^{+\infty} \int_0^s \int_0^{t_n} \cdots \int_{0}^{t_2} \xi(s-t_n) \prod_{i=2}^{n} \Phi(t_i-t_{i-1}) dt_1\cdots dt_n \\
&+  \mu \sum_{n=1}^{+\infty} \int_0^s \int_0^{t_n} \cdots \int_{0}^{t_2} \xi(s-t_n) \prod_{i=2}^{n} \Phi(t_i-t_{i-1}) \sum_{j=1}^n \left(\zeta(t-t_j)+\int_{t_j}^{t} \zeta(t-u)  \Psi(u-t_j) du \right) dt_1\cdots dt_n.
\end{align*}
Using the computations of Part (i) we identify that  
$$
\mu \sum_{n=1}^{+\infty} \int_0^s \int_0^{t_n} \cdots \int_{0}^{t_2} \xi(s-t_n) \prod_{i=2}^{n} \Phi(t_i-t_{i-1}) dt_1\cdots dt_n =\E\left[X_s^{\xi}\right].
$$
In addition, if $\zeta \equiv 1$, 
$$ \mu \left[\int_0^{t} \zeta(t-u) \left(1+\int_0^u \Psi(u-v) dv\right) du \right] = \E\left[X_t^{\zeta}\right].$$
In case $\zeta \equiv \Phi$ we have
\begin{align*}
&\mu \left[\int_0^{t} \zeta(t-u) \left(1+\int_0^u \Psi(u-v) dv\right) du \right]\\
&=\mu \left[\int_0^{t} \Phi(t-u) du + \int_0^t \Phi(t-u) \int_0^u \Psi(u-v) dv du \right] \\
&=\mu \left[\int_0^{t} \Phi(t-u) du + \int_0^t \int_v^t \Phi(t-u) \Psi(u-v) du dv \right] \\
&=\mu \left[\int_0^{t} \Phi(t-u) du + \int_0^t \int_0^{t-v} \Phi(t-v-w) \Psi(w) du dv \right] \\
&=\mu \left[\int_0^{t} \Phi(t-u) du + \int_0^t (\Psi \ast \Phi)(t-v)dv \right] \\
&=\mu \left[\int_0^{t} \Phi(t-u) du + \int_0^t \Psi(t-v)dv - \int_0^t \Phi(t-v)dv \right] \\
&=\mu \int_0^t \Psi(t-v)dv = \E\left[X_t^\zeta\right].
\end{align*}
where we have used the fact that $\Psi \ast \Phi = \Psi - \Phi$.
Hence
\begin{align*}
&\E\left[X_s^\xi X_t^\zeta\right] \\
&= \E\left[X_s^\xi\right] \E\left[X_t^\zeta\right] \\
&+  \mu \sum_{n=1}^{+\infty} \int_0^s \int_0^{t_n} \cdots \int_{0}^{t_2} \xi(s-t_n) \prod_{i=2}^{n} \Phi(t_i-t_{i-1}) \sum_{j=1}^n \left(\zeta(t-t_j)+\int_{t_j}^{t} \zeta(t-u)  \Psi(u-t_j) du \right) dt_1\cdots dt_n.
\end{align*}
We now deal with the term in $\mu$. We have 
\begin{align*}
&\sum_{n=1}^{+\infty} \int_0^s \int_0^{t_n} \cdots \int_{0}^{t_2} \xi(s-t_n) \prod_{i=2}^{n} \Phi(t_i-t_{i-1}) \sum_{j=1}^n \left(\zeta(t-t_j)+\int_{t_j}^{t} \zeta(t-u)  \Psi(u-t_j) du \right) dt_1\cdots dt_n\\
&= \sum_{j=1}^{+\infty} \sum_{n=j}^{+\infty} \int_0^s \int_0^{t_n} \cdots \int_{0}^{t_2} \xi(s-t_n) \prod_{i=2}^{n} \Phi(t_i-t_{i-1}) \left(\zeta(t-t_j)+\int_{t_j}^{t} \zeta(t-u)  \Psi(u-t_j) du \right) dt_1\cdots dt_n\\
&= \sum_{j=1}^{+\infty} \sum_{n=j}^{+\infty} \int_0^s \int_0^{t_n} \cdots \int_{0}^{t_2} \xi(s-t_n) \prod_{i=2}^{n} \Phi(t_i-t_{i-1}) \zeta(t-t_j) dt_1\cdots dt_n\\
&+ \sum_{j=1}^{+\infty} \sum_{n=j}^{+\infty} \int_0^s \int_0^{t_n} \cdots \int_{0}^{t_2} \xi(s-t_n) \prod_{i=2}^{n} \Phi(t_i-t_{i-1}) \int_{t_j}^{t} \zeta(t-u) \Psi(u-t_j) du dt_1\cdots dt_n\\
&=:T_1+T_2.
\end{align*}
We treat the two terms separately. We will make use for both terms of Fubini's theorem. For term $T_2$ the domain of integration is : 
$$ 0< t_1 < \cdots < t_j < t_{j+1}< \cdots < t_n < s; \; \quad t_j < u <t.$$
We rewrite this domain as 
$$ \int_0^t \int_0^{s\wedge u} \left(\int_{0<t_1<\cdots<t_j} dt_1\cdots dt_{j-1}\right) \left(\int_{t_j<t_{j+1}<\cdots<t_n<s} dt_{j+1}\cdots dt_{n}\right) dt_jdu.$$
Hence we have
\begin{align*}
&T_2\\
&=\sum_{j=1}^{+\infty} \sum_{n=j}^{+\infty} \int_0^s \int_0^{t_n} \cdots \int_{0}^{t_2} \xi(s-t_n) \prod_{i=2}^{n} \Phi(t_i-t_{i-1}) \int_{t_j}^{t} \zeta(t-u) \Psi(u-t_j) du dt_1\cdots dt_n\\
&= \sum_{j=1}^{+\infty} \sum_{n=j}^{+\infty} \int_0^t \int_0^{s\wedge u} \zeta(t-u) \Psi(u-t_j) \left(\int_{t_j}^s \xi(s-t_n) \int_{t_j}^{t_n} \cdots \int_{t_j}^{t_{j+2}} \prod_{k=j+1}^{n} \Phi(t_k-t_{k-1}) dt_{j+1} \cdots dt_{n-1} dt_n\right) \\
&\times \left(\int_0^{t_j} \int_0^{t_{j-1}} \cdots \int_0^{t_2} \prod_{\ell=2}^{j} \Phi(t_\ell-t_{\ell-1}) dt_1 \cdots dt_{j-2} dt_{j-1}\right) dt_j du.
\end{align*}
By Lemma \ref{lemma:magic},
$$  \int_0^{t_j} \int_0^{t_{j-1}} \cdots \int_0^{t_2} \prod_{\ell=2}^{j} \Phi(t_\ell-t_{\ell-1}) dt_1 \cdots dt_{j-2} dt_{j-1} = \int_0^{t_j} \Phi_{j-1}(t_j-w) dw$$
and
$$
 \int_{t_j}^s \xi(s-t_n) \int_{t_j}^{t_n} \cdots \int_{t_j}^{t_{j+2}} \prod_{k=j+1}^{n} \Phi(t_k-t_{k-1}) dt_{j+1} \cdots dt_{n-1} dt_n \hspace{6cm}$$
  \vspace{-0.5cm}
 $$
  \hspace{4cm} =
\left\{
\begin{array}{ll}
  \int_{t_j}^s \Phi_{n-j}(y-t_j) dy, \; &\textrm{ if } \xi \equiv 1  \\
\Phi_{n-j+1}(s-t_j), \; &\textrm{ if } \xi \equiv \Phi \textrm{ and } j \geq n+1;\\
\xi(s-t_j), \; &\textrm{ if } \xi \equiv \Phi \textrm{ and } j = n.
\end{array}
\right.
$$
$$ \mbox{ In addition,  recalling that }   \int \Phi_0 = 1,   \mbox{ we have} \quad   \sum_{j=0}^{+\infty} \int \Phi_{j}(x) dx = 1+\int \Psi(x) dx.  \hspace{4cm} $$
\begin{itemize}
\item If $\xi \equiv 1$
\begin{align*}
T_2
&= \sum_{j=1}^{+\infty} \sum_{n=j}^{+\infty} \int_0^t \int_0^{s\wedge u} \zeta(t-u) \Psi(u-t_j) \int_{t_j}^s \Phi_{n-j}(y-t_j) dy \times \int_0^{t_j} \Phi_{j-1}(t_j-w) dw dt_j du \\
&= \sum_{j=1}^{+\infty} \sum_{n=j}^{+\infty} \int_0^t \int_0^{s\wedge u} \zeta(t-u) \Psi(u-v) \int_{v}^s \Phi_{n-j}(y-v) dy \times \int_0^{v} \Phi_{j-1}(v-w) dw dv du \\
&= \sum_{j=1}^{+\infty} \int_0^t \int_0^{s\wedge u} \zeta(t-u) \Psi(u-v) \left(1+\int_{v}^s \Psi(y-v) dy \right) \times \int_0^{v} \Phi_{j-1}(v-w) dw dv du \\
&= \int_0^t \int_0^{s\wedge u} \zeta(t-u) \Psi(u-v) \left(1+\int_{v}^s \Psi(y-v) dy \right) \left( 1+\int_0^{v} \Psi(v-w) dw\right) dv du.
\end{align*}

\item If $\xi \equiv \Phi$,
\begin{align*}
T_2
&= \sum_{j=1}^{+\infty} \sum_{n=j}^{+\infty} \int_0^t \int_0^{s\wedge u} \zeta(t-u) \Psi(u-t_j) \\
&\times \left(\int_{t_j}^s \xi(s-t_n) \int_{t_j}^{t_n} \cdots \int_{t_j}^{t_{j+2}} \prod_{k=j+1}^{n} \Phi(t_k-t_{k-1}) dt_{j+1} \cdots dt_{n-1} dt_n\right) \\
&\times \left(\int_0^{t_j} \int_0^{t_{j-1}} \cdots \int_0^{t_2} \prod_{\ell=2}^{j} \Phi(t_\ell-t_{\ell-1}) dt_1 \cdots dt_{j-2} dt_{j-1}\right) dt_j du\\
&= \sum_{j=1}^{+\infty} \int_0^t \int_0^{s\wedge u} \zeta(t-u) \Psi(u-t_j) \Phi(s-t_j) \int_0^{t_j} \Phi_{j-1}(t_j-w) dw dt_j du\\
&+ \sum_{j=1}^{+\infty} \sum_{n=j+1}^{+\infty} \int_0^t \int_0^{s\wedge u} \zeta(t-u) \Psi(u-t_j) \Phi_{n-j+1}(s-t_j) \int_0^{t_j} \Phi_{j-1}(t_j-w) dw dt_j du\\
&= \sum_{j=1}^{+\infty} \int_0^t \int_0^{s\wedge u} \zeta(t-u) \Psi(u-t_j) \Phi(s-t_j) \int_0^{t_j} \Phi_{j-1}(t_j-w) dw dt_j du\\
&+ \sum_{j=1}^{+\infty} \int_0^t \int_0^{s\wedge u} \zeta(t-u) \Psi(u-t_j) (\Psi(s-t_j)-\Phi(s-t_j)) \int_0^{t_j} \Phi_{j-1}(t_j-w) dw dt_j du\\
&= \sum_{j=1}^{+\infty} \int_0^t \int_0^{s\wedge u} \zeta(t-u) \Psi(u-t_j) \Psi(s-t_j) \int_0^{t_j} \Phi_{j-1}(t_j-w) dw dt_j du\\
&= \sum_{j=1}^{+\infty} \int_0^t \int_0^{s\wedge u} \zeta(t-u) \Psi(u-v) \Psi(s-v) \int_0^{v} \Phi_{j-1}(v-w) dw dv du\\
&= \int_0^t \int_0^{s\wedge u} \zeta(t-u) \Psi(u-v) \Psi(s-v) \left(1+\int_0^{v} \Psi(v-w) dw\right) dv du.
\end{align*}

\end{itemize}

Using once again Fubini's theorem we get 

\begin{align*}
T_1
&=\sum_{j=1}^{+\infty} \sum_{n=j}^{+\infty} \int_0^s \int_0^{t_n} \cdots \int_{0}^{t_2} \xi(s-t_n) \prod_{i=2}^{n} \Phi(t_i-t_{i-1}) \zeta(t-t_j) dt_1\cdots dt_n\\
&= \sum_{j=1}^{+\infty} \sum_{n=j}^{+\infty} \int_0^s \zeta(t-t_j) \left(\int_{t_j}^s \xi(s-t_n) \int_{t_j}^{t_n} \cdots \int_{t_j}^{t_{j+2}} \prod_{k=j+1}^{n} \Phi(t_k-t_{k-1}) dt_{j+1} \cdots dt_{n-1} dt_n\right) \\
&\times \left(\int_0^{t_j} \int_0^{t_{j-1}} \cdots \int_0^{t_2} \prod_{\ell=2}^{j} \Phi(t_\ell-t_{\ell-1}) dt_1 \cdots dt_{j-2} dt_{j-1}\right) dt_j.
\end{align*}

\begin{itemize}
\item If $\xi \equiv 1$
\begin{align*}
T_1
&=\sum_{j=1}^n \sum_{n=j}^{+\infty} \int_0^s \int_0^{t_n} \cdots \int_{0}^{t_2} \prod_{i=2}^{n} \Phi(t_i-t_{i-1}) \zeta(t-t_j) dt_1\cdots dt_n\\
&= \sum_{j=1}^n \sum_{n=j}^{+\infty}\int_0^s \zeta(t-t_j) \int_{t_j}^s \Phi_{n-j}(y-t_j) dy\int_0^{t_j} \Phi_{j-1}(t_j-w) dw dt_j\\
&= \sum_{j=1}^n \sum_{n=j}^{+\infty}\int_0^s \zeta(t-r) \int_{r}^s \Phi_{n-j}(y-r) dy\int_0^{r} \Phi_{j-1}(r-w) dw dr\\
&= \int_0^s \zeta(t-u) \left( 1+\int_{u}^s \Psi(y-u) dy\right) \left(1+\int_0^{u} \Psi(w) dw\right) du.
\end{align*}
\item If $\xi \equiv \Phi$
\begin{align*}
T_1
&=\sum_{j=1}^{+\infty} \sum_{n=j}^{+\infty} \int_0^s \zeta(t-t_j) \left(\int_{t_j}^s \xi(s-t_n) \int_{t_j}^{t_n} \cdots \int_{t_j}^{t_{j+2}} \prod_{k=j+1}^{n} \Phi(t_k-t_{k-1}) dt_{j+1} \cdots dt_{n-1} dt_n\right) \\
&\times \left(\int_0^{t_j} \int_0^{t_{j-1}} \cdots \int_0^{t_2} \prod_{\ell=2}^{j} \Phi(t_\ell-t_{\ell-1}) dt_1 \cdots dt_{j-2} dt_{j-1}\right) dt_j \\
&=\sum_{j=1}^{+\infty} \int_0^s \zeta(t-t_j) \xi(s-t_j) \int_0^{t_j} \Phi_{j-1}(t_j-w) dw dt_j \\
&+\sum_{j=1}^{+\infty} \sum_{n=j+1}^{+\infty} \int_0^s \zeta(t-t_j) \Phi_{n-j+1}(s-t_j) \int_0^{t_j} \Phi_{j-1}(t_j-w) dw dt_j \\
&=\sum_{j=1}^{+\infty} \int_0^s \zeta(t-v) \xi(s-v) \int_0^{v} \Phi_{j-1}(v-w) dw dv \\
&+\sum_{j=1}^{+\infty} \sum_{n=j+1}^{+\infty} \int_0^s \zeta(t-t_j) \Phi_{n-j+1}(s-v) \int_0^{v} \Phi_{j-1}(v-w) dw dv \\
&=\sum_{j=1}^{+\infty} \int_0^s \zeta(t-v) \Phi(s-v) \int_0^{v} \Phi_{j-1}(v-w) dw dv \\
&+\sum_{j=1}^{+\infty} \int_0^s \zeta(t-v) (\Psi(s-v)-\Phi(s-v)) \int_0^{v} \Phi_{j-1}(v-w) dw dv \\
&=\int_0^s \zeta(t-v) \Psi(s-v) \left(1+\int_0^{v} \Psi(w) dw\right) dv.
\end{align*}
\end{itemize}
Putting together the terms $T_1$ and $T_2$ for $\xi \equiv 1$ (resp. $\xi \equiv \Phi$), we get Relation  \eqref{eq:mainbeta1} (resp.  Relation  \eqref{eq:mainbetaphi}). 
\end{proof}

\subsection{Proof of Theorem \ref{th:main}}
\label{subsection:proofof main}

This result is a direct consequence of Theorem \ref{th:main2}. More precisely, from Part (i) of Theorem \ref{th:main2} we recover the well-known expressions of the expectation of the marginals of the Hawkes process and of its intensity. Indeed (noting that $\lambda_t = \mu + X_t^\xi$ with $\xi\equiv \Phi$)
$$\E\left[H_t\right] = \mu t + \mu \int_0^t\int_0^u \Psi(r) dr du; \quad \E\left[\lambda_t\right] = \mu \left(1+\int_0^t\Psi(r) dr\right); \quad t\geq 0.$$

\noindent
\textbf{Proof of Part (i):\\\\}
\noindent
We apply Relation (\ref{eq:mainbeta1}) with $\zeta \equiv 1$. We have for any $s\leq t$,

\begin{align*}
\E\left[H_s H_s\right]
&= \E\left[H_s\right] \E\left[H_t\right] \nonumber\\
&+ \mu \int_0^s \left( 1+\int_{u}^s \Psi(y-u) dy\right) \left(1+\int_0^{u} \Psi(w) dw\right) du \nonumber\\
&+\mu \int_0^t \int_0^{s\wedge u} \Psi(u-v) \left(1+\int_0^{v} \Psi(w) dw\right) \left(1+\int_{v}^s \Psi(y-v) dy\right) dv du \\
&= \E\left[H_s\right] \E\left[H_t\right] \nonumber\\
&+ \mu \int_0^s \left( 1+\int_{u}^s \Psi(y-u) dy\right) \left(1+\int_0^{u} \Psi(w) dw\right) du \nonumber\\
&+\mu \int_0^s \int_0^{u} \Psi(u-v) \left(1+\int_0^{v} \Psi(w) dw\right) \left(1+\int_{v}^s \Psi(y-v) dy\right) dv du \\
&+\mu \int_s^t \int_0^{s} \Psi(u-v) \left(1+\int_0^{v} \Psi(w) dw\right) \left(1+\int_{v}^s \Psi(y-v) dy\right) dv du \\
&= \E\left[H_s\right] \E\left[H_t\right] \nonumber\\
&+ \mu \int_0^s \left( 1+\int_{u}^s \Psi(y-u) dy\right) \left(1+\int_0^{u} \Psi(w) dw\right) du \nonumber\\
&+\mu \int_0^s \left(1+\int_0^{v} \Psi(w) dw\right) \left(1+\int_{v}^s \Psi(y-v) dy\right) \int_v^{s} \Psi(u-v) du dv \\
&+\mu \int_0^{s} \left(1+\int_0^{v} \Psi(w) dw\right) \left(1+\int_{v}^s \Psi(y-v) dy\right) \int_s^t \Psi(u-v) du dv \\
&= \E\left[H_s\right] \E\left[H_t\right] \nonumber\\
&+ \mu \int_0^s \left( 1+\int_{u}^s \Psi(y-u) dy\right) \left(1+\int_0^{u} \Psi(w) dw\right) du \nonumber\\
&+\mu \int_0^s \left(1+\int_0^{v} \Psi(w) dw\right) \left(1+\int_{v}^s \Psi(y-v) dy\right) \int_v^{t} \Psi(u-v) du dv \\
&= \E\left[H_s\right] \E\left[H_t\right] \nonumber\\
&+\mu \int_0^s \left(1+\int_0^{v} \Psi(w) dw\right) \left(1+\int_{v}^s \Psi(y-v) dy\right) \left(1+\int_v^{t} \Psi(u-v) du\right) dv.
\end{align*}

\noindent
\textbf{Proof of Part (ii):\\\\}
\noindent
Let $s\leq t$, once should compute using Theorem \ref{th:main2}: $\E\left[X_s^\xi X_t^\zeta\right]$ first with $(\zeta,\xi) \equiv (1,\Phi)$ and then with $(\zeta,\xi) \equiv (\Phi,1)$. Obviously both quantities are similar. We have by (\ref{eq:mainbetaphi})
\begin{align*}
&\E[(\lambda_s-\mu) H_t] - (\E[\lambda_s]-\mu) \E[H_t]\\
&=\mu \int_0^s \Psi(s-v) \left(1+\int_0^{v} \Psi(w) dw\right) dv \nonumber\\
&+\mu \int_0^t \int_0^{s\wedge u} \Psi(u-v) \Psi(s-v) \left(1+\int_0^{v} \Psi(v-w) dw\right) dv du\\
&=\mu \int_0^s \Psi(s-v) \left(1+\int_0^{v} \Psi(w) dw\right) dv \nonumber\\
&+\mu \int_0^s \int_0^{u} \Psi(u-v) \Psi(s-v) \left(1+\int_0^{v} \Psi(w) dw\right) dv du\\
&+\mu \int_s^t \int_0^{s} \Psi(u-v) \Psi(s-v) \left(1+\int_0^{v} \Psi(w) dw\right) dv du\\
&=\mu \int_0^s \Psi(s-v) \left(1+\int_0^{v} \Psi(w) dw\right) \left(1+ \int_v^{s} \Psi(u-v) du\right) dv\\
&+\mu \int_0^{s} \Psi(s-v)  \int_s^t \Psi(u-v) \left(1+\int_0^{v} \Psi(w) dw\right) du dv\\
&=\mu \int_0^s \Psi(s-v) \left(1+\int_0^{v} \Psi(w) dw\right) \left(1+ \int_v^{t} \Psi(u-v) du\right) dv.
\end{align*}

Thus
$$\E[\lambda_s H_t] - \E[\lambda_s] \E[H_t]=\mu \int_0^s \Psi(s-v) \left(1+\int_0^{v} \Psi(w) dw\right) \left(1+ \int_v^{t} \Psi(u-v) du\right) dv.$$

Similarly using (\ref{eq:mainbeta1})
\begin{align*}
&\E[H_s (\lambda_t-\mu)] - \E[H_s] (\E[\lambda_t]-\mu)\\
&=\mu  \int_0^s \Phi(t-v) \left( 1+\int_{v}^s \Psi(y-v) dy\right) \left(1+\int_0^{v} \Psi(w) dw\right) dv \nonumber\\
&+\mu \int_0^t \int_0^{s\wedge u} \Phi(t-u) \Psi(u-v) \left(1+\int_0^{v} \Psi(w) dw\right) \left(1+\int_{v}^s \Psi(y-v) dy\right) dv du \\
&=\mu \int_0^s \Phi(t-v) \left( 1+\int_{v}^s \Psi(y-v) dy\right) \left(1+\int_0^{v} \Psi(w) dw\right) dv \nonumber\\
&+\mu \int_0^s \int_0^{u} \Phi(t-u) \Psi(u-v) \left(1+\int_0^{v} \Psi(w) dw\right) \left(1+\int_{v}^s \Psi(y-v) dy\right) dv du \\
&+\mu \int_s^t \int_0^{s} \Phi(t-u) \Psi(u-v) \left(1+\int_0^{v} \Psi(w) dw\right) \left(1+\int_{v}^s \Psi(y-v) dy\right) dv du \\
&=\mu \int_0^s \Phi(t-v) \left( 1+\int_{v}^s \Psi(y-v) dy\right) \left(1+\int_0^{v} \Psi(w) dw\right) dv \nonumber\\
&+\mu \int_0^s  \left(1+\int_0^{v} \Psi(w) dw\right) \left(1+\int_{v}^s \Psi(y-v) dy\right) \int_v^{s} \Phi(t-u) \Psi(u-v) du dv \\
&+\mu \int_0^{s} \left(1+\int_0^{v} \Psi(w) dw\right) \left(1+\int_{v}^s \Psi(y-v) dy\right) \int_s^t \Phi(t-u) \Psi(u-v)  du dv \\
&=\mu \int_0^s \Phi(t-v) \left( 1+\int_{v}^s \Psi(y-v) dy\right) \left(1+\int_0^{v} \Psi(w) dw\right) dv \nonumber\\
&+\mu \int_0^s  \left(1+\int_0^{v} \Psi(w) dw\right) \left(1+\int_{v}^s \Psi(y-v) dy\right) \int_0^{t-v} \Phi(t-v-x) \Psi(x) dx dv \\
&=\mu \int_0^s \Phi(t-v) \left( 1+\int_{v}^s \Psi(y-v) dy\right) \left(1+\int_0^{v} \Psi(w) dw\right) dv \nonumber\\
&+\mu \int_0^s  \left(1+\int_0^{v} \Psi(w) dw\right) \left(1+\int_{v}^s \Psi(y-v) dy\right) [\Psi(t-v)-\Phi(t-v)] dv \\
&=\mu \int_0^s \Psi(t-v) \left(1+\int_0^{v} \Psi(w) dw\right) \left(1+\int_{v}^s \Psi(y-v) dy\right) dv.
\end{align*}
Therefore
$$\E[H_s \lambda_t] - \E[H_s] \E[\lambda_t]=\mu \int_0^s \Psi(t-v) \left(1+\int_0^{v} \Psi(w) dw\right) \left(1+\int_{v}^s \Psi(y-v) dy\right) dv.$$

\noindent
\textbf{Proof of Part (iii):\\\\}
\noindent
Let $s\leq t$. Relation (\ref{eq:mainbetaphi}) entails that 
\begin{align*}
&\E\left[(\lambda_s-\mu) (\lambda_t-\mu)\right]\\
&= \E\left[(\lambda_s-\mu)\right] \E\left[(\lambda_t-\mu)\right] \\
&+\mu \int_0^s \Phi(t-v) \Psi(s-v) \left(1+\int_0^{v} \Psi(w) dw\right) dv\\
&+\mu \int_0^t \int_0^{s\wedge u} \Phi(t-u) \Psi(u-v) \Psi(s-v) \left(1+\int_0^{v} \Psi(v-w) dw\right) dv du \\
&= \E\left[(\lambda_s-\mu)\right] \E\left[(\lambda_t-\mu)\right] \\
&+\mu \int_0^s \Phi(t-v) \Psi(s-v) \left(1+\int_0^{v} \Psi(w) dw\right) dv\\
&+\mu \int_0^s \Psi(s-v) \left(1+\int_0^{v} \Psi(v-w) dw\right) \int_v^{s} \Phi(t-u) \Psi(u-v) du dv \\
&+\mu \int_0^{s} \Psi(s-v) \left(1+\int_0^{v} \Psi(v-w) dw\right) \int_s^t \Phi(t-u) \Psi(u-v) du dv \\
&= \E\left[(\lambda_s-\mu)\right] \E\left[(\lambda_t-\mu)\right] \\
&+\mu \int_0^s \Phi(t-v) \Psi(s-v) \left(1+\int_0^{v} \Psi(w) dw\right) dv\\
&+\mu \int_0^s \Psi(s-v) \left(1+\int_0^{v} \Psi(v-w) dw\right) \int_v^{t} \Phi(t-u) \Psi(u-v) du dv \\
&= \E\left[(\lambda_s-\mu)\right] \E\left[(\lambda_t-\mu)\right] \\
&+\mu \int_0^s \Phi(t-v) \Psi(s-v) \left(1+\int_0^{v} \Psi(w) dw\right) dv\\
&+\mu \int_0^s \Psi(s-v) \left(1+\int_0^{v} \Psi(v-w) dw\right) \int_0^{t-v} \Phi(t-v-x) \Psi(x) dx dv \\
&= \E\left[(\lambda_s-\mu)\right] \E\left[(\lambda_t-\mu)\right] \\
&+\mu \int_0^s \Psi(s-v) \Psi(t-v) \left(1+\int_0^{v} \Psi(v-w) dw\right) dv.
\end{align*}

Hence
$$\E\left[\lambda_s \lambda_t\right] - \E\left[\lambda_s\right] \E\left[\lambda_t\right]=\mu \int_0^s \Psi(s-v) \Psi(t-v) \left(1+\int_0^{v} \Psi(v-w) dw\right) dv.$$

%\bibliographystyle{plain}
%\bibliography{biblioHR}

\end{document}